\documentclass[12pt]{amsart}
\usepackage{amsmath, amsthm, amscd, amsfonts}

\makeatletter \oddsidemargin.9375in \evensidemargin \oddsidemargin
\newtheorem{theorem}{Theorem}[section]

\newtheorem{proposition}[theorem]{Proposition}

\theoremstyle{definition}
\newtheorem{definition}[theorem]{Definition}

\theoremstyle{remark}

\numberwithin{equation}{section}

\begin{document}

\title[Pursuit-Evasion Game with Hybrid System of Dynamics]
{Pursuit-Evasion Game with Hybrid System of Dynamics}
\author[Mehdi Salimi]{\textbf{Mehdi Salimi}\\
Center for Dynamics and Institute for Analysis, Department of
Mathematics, Technische Universit\"{a}t Dresden, 01062 Dresden,
Germany}

\keywords{Differential Game; Hybrid System; Pursuer;
Evader; Winning Strategy. \\
\indent \textit{E-mail addresses:} mehdi.salimi@tu-dresden.de}


 \maketitle





\section{Formulation of the problem and result}

\par{In the space $l_2$ consisting of elements
$a=(a_1,a_2,\ldots,a_m,\ldots),$ with $\sum_{m=1}^\infty
a_m^2<\infty$, and inner product $(a,b)=\sum_{m=1}^\infty a_m
b_m$, the motions of the countably many pursuers $P_i$ and the
evader $E$ are defined by the hybrid system of differential
equations
\begin{equation}\label{aa}
 \begin{split}
&P:\dot{p}=\mu, \quad p(0)=p_0, \\
&E:\ddot{e}=\nu, \quad \dot{e}(0)=e^1, \quad e(0)=e^0,
 \end{split}
\end{equation}
where $p,p_0,\mu,e,e^0,e^1,\nu\in l_2,$
$\mu=(\mu_{1},\mu_{2},\ldots,\mu_{m},\ldots)$ is the control
parameter of the pursuer $P_i$, and
$\nu=(\nu_1,\nu_2,\ldots,\nu_m,\ldots)$ is that of the evader $E$.
Let $\varphi$ be a given positive number.

A ball of radius $\delta$ and center at the point $x_0$ is denoted
by $B(x_0,\delta)=\{x\in l_2:\|x-x_0\|\le \delta\}$.
\begin{definition}
A function $\mu(\cdot),$ $\mu:[0,\varphi]\to l_2,$ such that
$\mu_{m}:[0,\varphi]\to R^1,$ $m=1,2,\ldots,$ are Borel measurable
functions and
$$
\|\mu(\cdot)\|^{2}_2=\int\limits_0^\varphi\|\mu(s)\|^2\,ds\le\Gamma^{2},
\quad\|\mu\|=\sum_{m=1}^\infty \mu_{m}^2,
$$
where $\Gamma$ is given positive number, is called an
\textit{admissible control of the pursuer}.
\end{definition}
\begin{definition}
A function $\nu(\cdot),$ $\nu:[0,\varphi]\to l_2,$ such that
$\nu_m:[0,\varphi]\to R^1,$ ${m=1,2,\ldots,}$ are Borel measurable
functions and
$$
\|\nu(\cdot)\|^{2}_2=\int\limits_0^\varphi\|\nu(s)\|^2\,ds\le\Upsilon^{2},
\quad\|\nu\|^2=\sum_{m=1}^\infty \nu_{m}^2,
$$
where $\Upsilon$ is a given positive number, is called an
\textit{admissible control of the evader}.
\end{definition}
Once the players' admissible controls $\mu(\cdot)$ and
$\nu(\cdot)$ are chosen, the corresponding motions $p(\cdot)$ and
$e(\cdot)$ of the players are defined as
$$
p(t)=(p_{1}(t),p_{2}(t),\ldots,p_{m}(t),\ldots), \quad
e(t)=(e_1(t),e_2(t),\ldots,e_k(t),\ldots),
$$
\begin{equation*}
p_{m}(t)=p_{m}^0+\int \limits_0^t\mu_{m}(s)\,ds, \quad
e_m(t)=e_m^0+te_m^1+\int\limits_0^t\int\limits_0^s\nu_m(r)\,dr\,ds.
\end{equation*}

One could observe that $p(\cdot),e(\cdot)\in C(0,\varphi;l_2),$
where $C(0,\varphi;l_2)$ is the space of functions $$f(t)
=(f_1(t),f_2(t),\ldots,f_m(t),\ldots)\in l_2, \quad t\ge0,$$ such
that the following properties are valid.\\
(1)~$f_m(t),$ $0\le t\le\varphi,$ $m=1,2,\ldots,$ are absolutely
continuous
functions;\\
(2)~$f(t),$ $0\le t\le\varphi,$ is a continuous function in the
norm of ~$l_2.$
\begin{definition}
A function $\Xi(t,p,e,\nu),$ $\Xi:[0,\infty)\times l_2\times
l_2\times l_2\to l_2,$ such that the system
\begin{equation*}
 \begin{split}
\dot{p}&=\Xi(t,p,e,\nu), \quad p(0)=p^0,\\
\ddot{e}&=\nu, \quad e(0)=e^0, \quad \dot{e}(0)=e^1,
 \end{split}
\end{equation*}
has a unique solution $(p(\cdot),e(\cdot)),$ with
$p(\cdot),e(\cdot)\in C(0,\varphi;l_2),$ for an arbitrary
admissible control $\nu=\nu(t),$ $0\le t\le\varphi,$ of the evader
$E,$ is called a \textit{strategy of the pursuer~$P$}. A strategy
$\Xi$ is said to be \textit{admissible} if each control formed by
this strategy is admissible.
\end{definition}
For the admissible control
$\nu(t)=(\nu_{1}(t),\nu_{2}(t),\ldots),~ 0\leq t\leq \varphi$, of
the evader $E$, according to (\ref{aa}) we have
\begin{equation*}
e(\varphi)=e^0+e^1\varphi+\int\limits_0^\varphi\int\limits_0^t
\nu(s)\,ds
\,dt=e^0+e^1\varphi+\int\limits_0^\varphi(\varphi-t)\nu(t)\,dt,
\end{equation*}
and using (\ref{a}) one can see
\begin{equation*}
e(\varphi)=e_{0}+\int\limits_0^\varphi(\varphi-t)\nu(t)\,dt=
e^0+e^1\varphi+\int\limits_0^\varphi(\varphi-t)\nu(t)\,dt.
\end{equation*}
Therefore, instead of differential game described by (\ref{aa}) we
can use an equivalent differential game with the same payoff
function as the following:
\begin{equation}\label{a}
 \begin{split}
&P:\dot{p}(t)=\mu(t), \quad p(0)=p_{0},\\
&E:\dot{e}(t)=(\varphi-t)\nu(t), \quad e(0)=e_0=e^1\varphi+e^0.
 \end{split}
\end{equation}
\begin{proposition}
The attainability domain of the pursuer $P$ at time  $\varphi$
from the initial state $p_{0}$ at time $t_0=0$ is the closed ball
$B\left(p_{0},\Gamma \sqrt \varphi \right).$
\end{proposition}
\begin{proof}
By Cauchy-Schwartz inequality we obtain
\begin{equation*}
 \begin{split}
\|p(\varphi)-p_{0}\|&=\biggl\|\int\limits_0^\varphi
\mu(s)\,ds\biggr\|\\
&\le\int\limits_0^\varphi\|\mu(s)\|\,ds\\
&\le\left(\int\limits_0^\varphi1^2\,ds\right)^{1/2}.\left(\int\limits_0^\varphi\|\mu(s)\|^2\,ds\right)^{1/2}
\le\Gamma \sqrt \varphi.
 \end{split}
\end{equation*}
Let $\bar p\in B\left(p_{0},\Gamma \sqrt \varphi\right).$ If the
pursuer $P$ uses the control

$$
\mu(t)=\frac{\bar p-p_{0}}{\varphi}, \quad 0\leq t\leq\varphi,
$$
then we obtain

\begin{equation*}
p(\varphi)=p_{0}+\int\limits_0^\varphi\mu(t)\,dt=p_0+\int\limits_0^\varphi
\frac{\bar p-p_{0}}{\varphi}\,dt=p_{0}+\bar p-p_{0}=\bar p.
\end{equation*}
The above pursuer's control is admissible. Indeed,
$$
\int\limits_0^\varphi\|\mu(t)\|^2\,dt=\int\limits_0^\varphi\|
\frac{\bar p-p_{0}}{\varphi} \|^2\,dt=\frac{\|\bar p-p_{0}\|^2}
{\varphi}\leq\frac{1}{\varphi}\Gamma^2\varphi=\Gamma^2.
$$
\end{proof}

\begin{proposition}
The attainability domain of the evader $E$ at time  $\varphi$ from
the initial state $e_{0}$ at time $t_0=0$ is the closed ball
$B\left(e_{0},\Upsilon \sqrt \frac{\varphi^3}{3} \right).$
\end{proposition}
\begin{proof}
We have
\begin{equation*}
 \begin{split}
\|e(\varphi)-e_{0}\|&=\biggl\|\int\limits_0^\varphi
(\varphi-t)\nu(t)\,dt\biggr\|\\
&\le\int\limits_0^\varphi\|(\varphi-t)\nu(t)\|\,dt\\
&\le\left(\int\limits_0^\varphi(\varphi-t)^2\,dt\right)^{1/2}.\left(\int\limits_0^\varphi\|\nu(t)\|^2\,dt\right)^{1/2}
\le\Upsilon \sqrt \frac{\varphi^3}{3}.
 \end{split}
\end{equation*}
Let $\bar e\in B\left(e_{0},\Upsilon \sqrt
\frac{\varphi^3}{3}\right).$ If the evader $E$ uses the control

$$
\nu(t)=3(\varphi-t)\frac{\bar e-e_{0}}{\varphi^3}, \quad 0\leq
t\leq\varphi,
$$
then we obtain

\begin{equation*}
e(\varphi)=e_{0}+\int\limits_0^\varphi(\varphi-t)\nu(t)\,dt=e_0+3\frac{\bar
e-e_{0}}{\varphi^3}\int\limits_0^\varphi
(\varphi-t)^2\,dt=e_{0}+\bar e-e_{0}=\bar e.
\end{equation*}
The above evader's control is admissible. Indeed,
$$
\int\limits_0^\varphi\|\nu(t)\|^2\,dt=\int\limits_0^\varphi\|
3(\varphi-t)\frac{\bar e-e_{0}}{\varphi^3}\|^2\,dt=9\frac{\|\bar
e-e_{0}\|^2}
{\varphi^6}\int\limits_0^\varphi(\varphi-t)^2\,dt\leq\frac{9}{\varphi^6}\Upsilon^2\frac{\varphi^3}{3}\frac{\varphi^3}{3}=\Upsilon^2.
$$
\end{proof}

The problem is to construct a winning strategy for the pursuer in
the game (\ref{aa}) that guarantees the equality
$p(\varphi)=e(\varphi)$, for any admissible control of the evader.

\begin{theorem}
Let
$$
Z=\biggl\{\zeta\in
l_2:2(e_0-p_0,\zeta)\le\varphi\left(\Gamma^2-\Upsilon^2\sqrt{\frac{\varphi^5}{5}}\right)+\|e_0\|^2-\|p_0\|^2,~~
e_0\neq p_0\biggl\}.
$$
If $e(\varphi)\in Z$ (phase constraint), then the pursuer has a
winning strategy.
\end{theorem}

\begin{proof}
Let's define the following strategy as a winning strategy for the
pursuer.

$$
\Xi(t)=\frac{e_0-p_0}{\varphi}+(\varphi-t)\nu(t),~~0 \le
t\le\varphi.
$$
Let's show that the above strategy is admissible. Since the evader
is satisfied to the phase constraint, we have
$$
2(e_0-p_0,e(\varphi))\le\varphi\left(\Gamma^2-\Upsilon^2\sqrt{\frac{\varphi^5}{5}}\right)+\|e_0\|^2-\|p_0\|^2.
$$
Using above inequality we have the following:
\begin{eqnarray*}
2\left(e_0-p_0,\int\limits_0^\varphi(\varphi-t)\nu(t)\,dt\right)&\leq&\varphi\left(\Gamma^2-\Upsilon^2\sqrt{\frac{\varphi^5}{5}}\right)-\|e_0-p_0\|^2.
\end{eqnarray*}
Indeed,
\begin{eqnarray*}
2\left(e_0-p_0,\int\limits_0^\varphi(\varphi-t)\nu(t)\,dt\right)&=&2\left(e_0-p_0,e(\varphi)-e_0\right)\\
&=&2\left(e_0-p_0,e(\varphi)\right)-2\left(e_0-p_0,e_0\right)\\
&=&2\left(e_0-p_0,e(\varphi)\right)-2\|e_0\|^2+2(p_0,e_0)\\
&\leq&\varphi\left(\Gamma^2-\Upsilon^2\sqrt{\frac{\varphi^5}{5}}\right)+\|e_0\|^2-\|p_0\|^2-2\|e_0\|^2+2(p_0,e_0)\\
&\leq&\varphi\left(\Gamma^2-\Upsilon^2\sqrt{\frac{\varphi^5}{5}}\right)-\|e_0\|^2-\|p_0\|^2+2(p_0,e_0)\\
&\leq&\varphi\left(\Gamma^2-\Upsilon^2\sqrt{\frac{\varphi^5}{5}}\right)-\left(\|e_0\|^2+\|p_0\|^2-2(p_0,e_0)\right)\\
&\leq&\varphi\left(\Gamma^2-\Upsilon^2\sqrt{\frac{\varphi^5}{5}}\right)-\|e_0-p_0\|^2.
\end{eqnarray*}
Thus, taking contribution of above inequality
\begin{eqnarray*}
\int\limits_0^\varphi\|\Xi(t)\|^2\,dt&=&\int\limits_0^\varphi\|
\frac{e_0-p_0}{\varphi}+(\varphi-t)\nu(t)\|^2\,dt\\
&=&\int\limits_0^\varphi\left(\|
\frac{e_0-p_0}{\varphi}\|^2+2\left(\frac{e_0-p_0}{\varphi},(\varphi-t)\nu(t)\right)+\|(\varphi-t)\nu(t)\|^2\right)\,dt\\
&=&\int\limits_0^\varphi\frac{\|e_0-p_0\|^2}{\varphi^2}\,dt+2\int\limits_0^\varphi
\left(\frac{e_0-p_0}{\varphi},(\varphi-t)\nu(t)\right)\,dt+\int\limits_0^\varphi(\varphi-t)^2\|\nu(t)\|^2
\,dt\\
&\leq&\frac{\|e_0-p_0\|^2}{\varphi}+\frac{2}{\varphi}\left(e_0-p_0,\int\limits_0^\varphi(\varphi-t)\nu(t)\,dt\right)\\
&+&\left(\int\limits_0^\varphi(\varphi-t)^4\,dt\right)^{\frac{1}{2}}\left(\int\limits_0^\varphi\|\nu(t)\|^4\,dt\right)^{\frac{1}{2}}\\
&\leq&\frac{\|e_0-p_0\|^2}{\varphi}+\frac{1}{\varphi}\left(\varphi\left(\Gamma^2-\Upsilon^2\sqrt{\frac{\varphi^5}{5}}\right)-\|e_0-p_0\|^2\right)+\Upsilon^2\sqrt{\frac{\varphi^5}{5}}\\
&\leq&\Gamma^2,
\end{eqnarray*}
and therefore the strategy $\Xi$ is admissible.

Now we show that $\Xi$ is a winning strategy for the pursuer.
Indeed,
\begin{eqnarray*}
p(\varphi)&=&p_0+\int\limits_0^\varphi
\left(\frac{e_0-p_0}{\varphi}+(\varphi-t)\nu(t)\right)\,ds\\
&=&p_0+\int\limits_0^\varphi
\left(\frac{e_0-p_0}{\varphi}\right)\,ds+\int\limits_0^\varphi(\varphi-t)\nu(t)\,ds\\
&=&p_0+e_0-p_0+\int\limits_0^\varphi(\varphi-t)\nu(t)\,ds=e(\varphi).
\end{eqnarray*}
\end{proof}






\bibliographystyle{amsmath}
\bibliographystyle{amsmath}

\end{document}